\documentclass[11pt]{amsart}
\usepackage[margin=30mm]{geometry}
\usepackage{amsmath,amssymb}
\usepackage{amsthm}
\usepackage{mathrsfs}

\newtheorem{thm}{Theorem}[section]

\newtheorem{lem}{Lemma}[section]
\newtheorem{cor}{Corollary}[section]
\newtheorem{defi}{Definition}[section]
\newtheorem{ex}{Example}[section]
\newtheorem{rem}{Remark}[section]

\newtheorem*{theorem}{\it{Theorem}}

\usepackage{enumitem}

\begin{document}

\title{On ball expanding maps}
\author{Noriaki Kawaguchi}
\address{Department of Mathematical and Computing Science, School of Computing, Institute of Science Tokyo, 2-12-1 Ookayama, Meguro-ku, Tokyo 152-8552, Japan}
\email{gknoriaki@gmail.com}

\begin{abstract}
This paper studies the dynamical properties of ball expanding maps, a class of continuous self-maps defined on compact metric spaces. We present several basic properties of ball expanding maps, involving periodic points, chain recurrence, entropy, and mixing properties. We show that if a ball expanding map is transitive and $h$-expansive, then the map is locally one-to-one and so expanding; and if a ball expanding map is surjective, then its natural extension satisfies the L-shadowing property. Several examples are also given.
\end{abstract}

\maketitle

\markboth{NORIAKI KAWAGUCHI}{On ball expanding maps}

\section{Introduction}

{\em Expanding maps} have played a significant role in the study of dynamical systems due to their rich structure and chaotic properties (see, e.g., \cite{KH}). {\em Ball expanding maps} are a class of continuous self-maps defined on compact metric spaces, characterized by their local expansive behavior \cite{BGO}. In this paper, we investigate the dynamics of ball expanding maps, relating it to recurrence, entropy, mixing, and shadowing properties of dynamical systems. Our results reveal how local ball expansion imposes global dynamical constraints.

We begin with the definition of ball expanding maps. Throughout, $X$ denotes a compact metric space endowed with a metric $d\colon X\times X\to[0,\infty)$. For a subset $S$ of $X$ and $x\in X$, let $d(x,S)$ denote the distance of $x$ from $S$:
\[
d(x,S)=\inf_{y\in S}d(x,y).
\]
For $x\in X$ and $r>0$, we denote by $B_r(x)$ the closed $r$-ball centered at $x$:
\[
B_r(x)=\{y\in X\colon d(x,y)\le r\}.
\]
For a subset $S$ of $X$ and $r>0$, we denote by $B_r(S)$ the closed $r$-neighborhood of $S$:
\[
B_r(S)=\{x\in X\colon d(x,S)\le r\}.
\]

\begin{defi}
\normalfont
A continuous map $f\colon X\to X$ is said to be {\em ball expanding} if there are $0<L<1$ and $\delta_0>0$ such that
\[
B_\delta(f(x))\subset f(B_{L\delta}(x))
\]
for all $x\in X$ and $0<\delta\le\delta_0$ \cite{BGO}.
\end{defi}

\begin{rem}
\normalfont
In \cite{BGO}, a continuous map $f\colon X\to X$ is said to be {\em expanding} if there are $\delta_0>0$ and $0<L<1$ such that for any $x,y\in X$, $d(x,y)\le\delta_0$ implies
\[
d(f(x),f(y))\ge L^{-1}d(x,y).
\]
By Theorem 2.5 of \cite{BGO}, we know that for a continuous map $f\colon X\to X$, the following conditions are equivalent
\begin{itemize}
\item $f$ is expanding and open,
\item $f$ is ball expanding and locally one-to-one.
\end{itemize}
\end{rem}

Let us recall from \cite{BGO} the definition of h-shadowing. Given a continuous map $f\colon X\to X$ and $\delta>0$, a finite sequence $(x_i)_{i=0}^k$ of points in $X$, where $k\ge1$, is called a {\em $\delta$-chain} of $f$ if
\[
\sup_{0\le i\le k-1}d(f(x_i),x_{i+1})\le\delta.
\]

\begin{defi}
\normalfont
A continuous map $f\colon X\to X$ is said to have the {\em h-shadowing property} if for any $\epsilon>0$, there is $\delta>0$ such that every $\delta$-chain $(x_i)_{i=0}^k$ of $f$ satisfies
\[
\sup_{0\le i\le k-1}d(x_i,f^i(x))\le\epsilon
\]
and $x_k=f^k(x)$ for some $x\in X$.
\end{defi}

\begin{rem}
\normalfont
By Theorem 4.3 of \cite{BGO}, we know that if a continuous map $f\colon X\to X$ is ball expanding, then $f$ has the h-shadowing property. A map $f\colon X\to X$ is called an {\em open map} if $f(U)$ is open for all open subset $U$ of $X$. We easily see that a continuous map $f\colon X\to X$ with the h-shadowing property is an open map.
\end{rem}

In \cite{K3}, the ball expanding condition has been related to the contractive shadowing, i.e., the Lipschitz shadowing such that the Lipschitz constant is less than $1$. We recall the definition of (Lipschitz) shadowing.

\begin{defi}
\normalfont
Let $f\colon X\to X$ be a continuous map and let $\xi=(x_i)_{i\ge0}$ be a sequence of points in $X$. For $\delta>0$, $\xi$ is called a {\em $\delta$-pseudo orbit} of $f$ if
\[
\sup_{i\ge0}d(f(x_i),x_{i+1})\le\delta.
\]
For $\epsilon>0$, $\xi$ is said to be {\em $\epsilon$-shadowed} by $x\in X$ if
\[
\sup_{i\ge0}d(x_i,f^i(x))\le\epsilon.
\]
The map $f$ is said to have the {\em shadowing property} if for any $\epsilon>0$, there is $\delta>0$ such that every $\delta$-pseudo orbit of $f$ is $\epsilon$-shadowed by some point of $X$. For $L>0$, we say that $f$ has the {\em $L$-Lipschitz shadowing property} if there is $\delta_0>0$ such that for any $0<\delta\le\delta_0$, every $\delta$-pseudo orbit of $f$ is $L\delta$-shadowed by some point of $X$. We say that $f$ has the {\em contractive shadowing property} if $f$ has the $L$-Lipschitz shadowing property for some $0<L<1$ \cite{K3}.
\end{defi}

\begin{rem}
\normalfont
It is easy to see that if a continuous map $f\colon X\to X$ has the h-shadowing property, then $f$ satisfies the shadowing property.
\end{rem}

\begin{rem}
\normalfont
For a continuous map $f\colon X\to X$, let $0<L<1$ and $\delta_0>0$ be as in the definition of ball expanding. By Remark 1.3 (2) of \cite{K3}, we know that $f$ has the $M$-Lipschitz shadowing property, where $M=L/(1-L)$. Note that for any $i\ge0$, we have 
\[
B_\delta(f^i(x))\subset f^i(B_{L^i\delta}(x))
\]
for all $x\in X$ and $0<\delta\le\delta_0$. As a consequence, for every $i\ge1$, $f^i$ has the $M_i$-Lipschitz shadowing property, where $M_i=L^i/(1-L^i)$.
\end{rem}

A continuous map $f\colon X\to X$ is said to be {\em mixing} if for any non-empty open subsets $U,V$ of $X$, there is $i\ge0$ such that $f^j(U)\cap V\ne\emptyset$ for all $j\ge i$. We say that $f$ is {\em locally eventually onto} if every non-empty open subset $U$ of $X$ satisfies $f^i(U)=X$ for some $i\ge0$. We easily see that if $f$ is locally eventually onto, then $f$ is mixing, and the converse holds when $f$ has the h-shadowing property.

This paper is structured as follows.

\begin{itemize}
\item In Section 2, we present several basic dynamical properties of ball expanding maps. Specifically, for any ball expanding map, we show that
\begin{itemize}
\item[--] the set of periodic points is dense in the chain recurrent set (Theorem 2.1),
\item[--] if the topological entropy of the map is zero, then the chain recurrent set is finite (2.2),
\item[--] the map has only finitely many chain components (2.3),
\item[--] if the space is perfect, then the topological entropy of the map is positive (2.4),
\item[--] if the space is connected, then the map is locally eventually onto and hence mixing (2.5).
\end{itemize}
\item In Section 3, we provide several examples concerning ball expanding maps.
\item In Section 4, we show that if a ball expanding map is transitive and $h$-expansive, then the map is locally one-to-one and so expanding (Theorem 4.1).
\item In Section 5, we prove that if a ball expanding map is surjective, then its natural extension satisfies the L-shadowing property (Theorem 5.1). 
\item In Appendix A, we show that a similar conclusion as in Lemma 2.1 holds under the assumption of $L$-Lipschitz shadowing, where $0<L<1$.
\item In Appendix B, generalizing Theorem 2.5, we prove that for any continuous map with the h-shadowing property, if the space is connected, then the map is locally eventually onto and hence mixing.
\end{itemize}

\section{Basic dynamical properties}

In this section, we present several basic dynamical properties of ball expanding maps. {\em Chain components}, which appear in the (so-called) fundamental theorem of dynamical systems by Conley, are essential objects for a global understanding of dynamical systems \cite{C}. Our results and their proofs are based on the notions of chain recurrence and chain components. Their definitions are as follows.

\begin{defi}
\normalfont
Let $f\colon X\to X$ be a continuous map. For any $x,y\in X$, the notation $x\rightarrow y$ means that for every $\delta>0$,  there is a $\delta$-chain $(x_i)_{i=0}^k$ of $f$ with $x_0=x$ and $x_k=y$. The {\em chain recurrent set} $CR(f)$ for $f$ is defined by
\[
CR(f)=\{x\in X\colon x\rightarrow x\}.
\]
We define a relation $\leftrightarrow$ in
\[
CR(f)^2=CR(f)\times CR(f)
\]
by: for any $x,y\in CR(f)$, $x\leftrightarrow y$ if and only if $x\rightarrow y$ and $y\rightarrow x$. Note that $\leftrightarrow$ is a closed equivalence relation in $CR(f)^2$ and satisfies $x\leftrightarrow f(x)$ for all $x\in CR(f)$. An equivalence class $C$ of $\leftrightarrow$ is called a {\em chain component} for $f$. We denote by $\mathcal{C}(f)$ the set of chain components for $f$.
\end{defi}

The basic properties of chain components are given in the following remark.

\begin{rem}
\normalfont
The following properties hold
\begin{itemize}
\item $CR(f)=\bigsqcup_{C\in\mathcal{C}(f)}C$, a disjoint union,
\item every $C\in\mathcal{C}(f)$ is a closed subset of $CR(f)$ with $f(C)=C$,
\item for all $C\in\mathcal{C}(f)$, $f|_C\colon C\to C$ is chain transitive, i.e., for any $x,y\in C$ and $\delta>0$, there is a $\delta$-chain $(x_i)_{i=0}^k$ of $f|_C$ with $x_0=x$ and $x_k=y$.
\end{itemize}
\end{rem}

\begin{rem}
\normalfont
Let $f\colon X\to X$ be a continuous map. 
\begin{itemize}
\item[(1)] It holds that $CR(f^i)=CR(f)$ for all $i\ge1$. 
\item[(2)] The following conditions are equivalent
\begin{itemize}
\item $f|_{CR(f)}\colon CR(f)\to CR(f)$ is a homeomorphism,
\item $f^i|_{CR(f^i)}\colon CR(f^i)\to CR(f^i)$ is a homeomorphism for all $i\ge1$,
\item $f^i|_{CR(f^i)}\colon CR(f^i)\to CR(f^i)$ is a homeomorphism for some $i\ge1$.
\end{itemize}
\item[(3)] The following conditions are equivalent
\begin{itemize}
\item $\mathcal{C}(f)$ is a finite set,
\item $\mathcal{C}(f^i)$ is a finite set for all $i\ge1$,
\item $\mathcal{C}(f^i)$ is a finite set for some $i\ge1$.
\end{itemize}
\end{itemize}

(3) can be proved as follows. If $\mathcal{C}(f)$ is a finite set, then for all $i\ge1$, by $|\mathcal{C}(f^i)|\le i\cdot|\mathcal{C}(f)|$, $\mathcal{C}(f^i)$ is a finite set.  If $\mathcal{C}(f^i)$ is a finite set for some $i\ge1$, then by $|\mathcal{C}(f)|\le|\mathcal{C}(f^i)|$, $\mathcal{C}(f)$ is a finite set.
\end{rem}

Given a continuous map $f\colon X\to X$, we denote by $Per(f)$ the set of periodic points for $f$:
\[
Per(f)=\bigcup_{j\ge1}\{x\in X\colon f^j(x)=x\}.
\]
By Theorem 1.2 of \cite{K3}, we know that if $f$ has the $L$-Lipschitz shadowing property for some $0<L<1/2$, then $f$ satisfies $CR(f)=\overline{Per(f)}$. Let $0<L<1$ and $\delta_0>0$ be as in the definition of ball expanding.  By Remark 1.4, taking $i\ge1$ with $L^i<1/3$, we have that $f^i$ has the $M_i$-Lipschitz shadowing property, where $M_i=L^i/(1-L^i)<1/2$. It follows that
\[
CR(f^i)=\overline{Per(f^i)}.
\]
By $CR(f^i)=CR(f)$ and $Per(f^i)=Per(f)$, we obtain $CR(f)=\overline{Per(f)}$. In other words, we obtain the following theorem.

\begin{thm}
If a continuous map $f\colon X\to X$ is ball expanding, then $f$ satisfies $CR(f)=\overline{Per(f)}$.
\end{thm}

For a continuous map $f\colon X\to X$, we denote by $h_{\rm top}(f)$ the topological entropy of $f$ (see, e.g., \cite{W} for details). By results of \cite{M}, if $f$ has the shadowing property, then $f$ satisfies $h_{\rm top}(f)=0$ if and only if
\[
f|_{CR(f)}\colon CR(f)\to CR(f)
\]
is an equicontinuous homeomorphism. We recall some related results from \cite{K3} (see Theorems 1.4 and 1.7 of \cite{K3}).

\begin{theorem}
\normalfont
\begin{itemize}
\item If a homeomorphism $f\colon X\to X$ has the $L$-Lipschitz shadowing property for some $0<L<1$, then $X$ is a finite set.
\item For $L>0$, if a  continuous map $f\colon X\to X$ has the $L$-Lipschitz shadowing property, then so does $f|_{CR(f)}\colon CR(f)\to CR(f)$.
\end{itemize} 
\end{theorem}

From these facts, it follows that if a continuous map $f\colon X\to X$ has the $L$-Lipschitz shadowing property for some $0<L<1$, then the following conditions are equivalent
\begin{itemize}
\item $h_{\rm top}(f)=0$,
\item $CR(f)$ is a finite set,
\item $f|_{CR(f)}\colon CR(f)\to CR(f)$ is a homeomorphism.
\end{itemize}
Let $0<L<1$ and $\delta_0>0$ be as in the definition of ball expanding.  As in Remark 1.4, by taking $i\ge1$ with $L^i<1/2$, we have that $f^i$ has the $M_i$-Lipschitz shadowing property, where $M_i=L^i/(1-L^i)<1$. Note that $h_{\rm top}(f^i)=i\cdot h_{\rm top}(f)$. By the above equivalence, we obtain the following theorem (see also Remark 2.2).

\begin{thm}
If a continuous map $f\colon X\to X$ is ball expanding, then the following conditions are equivalent
\begin{itemize}
\item[(1)] $h_{\rm top}(f)=0$,
\item[(2)] $CR(f)$ is a finite set,
\item[(3)] $f|_{CR(f)}\colon CR(f)\to CR(f)$ is a homeomorphism.
\end{itemize}
\end{thm}

By Theorem 1.8 of \cite{K3}, we know that if $f$ has the $L$-Lipschitz shadowing property for some $0<L<1$, then $\mathcal{C}(f)$ is a finite set. By a similar argument as above, we obtain the following theorem (see also Remark 2.2).

\begin{thm}
If a continuous map $f\colon X\to X$ is ball expanding, then $\mathcal{C}(f)$ is a finite set.
\end{thm}

Following \cite{AHK}, we define the {\em terminal} chain components as follows.

\begin{defi}
\normalfont
Given a continuous map $f\colon X\to X$, we say that a closed subset $S$ of $X$ with $f(S)\subset S$ is {\em chain stable} if $x\rightarrow y$ implies $y\in S$ for all $x\in S$ and $y\in X$, or equivalently, for any $\epsilon>0$, there is $\delta>0$ such that every $\delta$-chain $(x_i)_{i=0}^k$ of $f$ with $x_0\in S$ satisfies $d(x_k,S)\le\epsilon$. We say that $C\in\mathcal{C}(f)$ is {\em terminal} if $C$ is chain stable. We denote by $\mathcal{C}_{\rm ter}(f)$ the set of terminal chain components for $f$.
\end{defi}

\begin{rem}
\normalfont
\begin{itemize}
\item For a continuous map $f\colon X\to X$ and $x\in X$, the {\em $\omega$-limit set} $\omega(x,f)$ of $x$ for $f$ is defined as the set of $y\in X$ such that $\lim_{j\to\infty}f^{i_j}(x)=y$ for some sequence $0\le i_1<i_2<\cdots$. Note that $\omega(x,f)$ is a closed subset of $X$ with $f(\omega(x,f))=\omega(x,f)$ and satisfies
\[
\lim_{i\to\infty}d(f^i(x),\omega(x,f))=0.
\]
Since we have $y\rightarrow z$ for all $y,z\in\omega(x,f)$, there is an unique $C(x,f)\in\mathcal{C}(f)$ such that $\omega(x,f)\subset C(x,f)$ and so
\[
\lim_{i\to\infty}d(f^i(x),C(x,f))=0.
\]
\item We can show that every continuous map $f\colon X\to X$ satisfies $\mathcal{C}_{\rm ter}(f)\ne\emptyset$ (see, e.g., Lemma 2.1 and Corollary 2.1 of \cite{K1}), but this fact is fairly obvious when $\mathcal{C}(f)$ is a finite set.
\end{itemize}
\end{rem}

In order to prove Theorems 2.4 and 2.5, let us first prove the following lemma.

\begin{lem}
Let $f\colon X\to X$ be a continuous map. Let $0<L<1$ and $\delta_0>0$ be as in the definition of ball expanding. For any $x\in X$ and any sequence $(y_i)_{i\le0}$ of points in $X$, if
\begin{itemize}
\item $d(y_0,x)\le\delta_0$,
\item $y_i=f(y_{i-1})$ for all $i\le0$,
\end{itemize}
then there is a sequence $(x_i)_{i\le0}$ of points in $X$ such that
\begin{itemize}
\item $x_0=x$,
\item $x_i=f(x_{i-1})$ for all $i\le0$,
\item $d(y_{-i},x_{-i})\le L^i\delta$ for all $i\ge0$, where $\delta=d(y_0,x)$.
\end{itemize}
\end{lem}

\begin{proof}
Letting $x_0=x$, we have $d(y_0,x_0)\le\delta$. Assume that $d(y_{-i},x_{-i})\le L^i\delta$ for some $i\ge0$ and $x_{-i}\in X$. Since $L^i\delta\le\delta\le\delta_0$, we have
\[
x_{-i}\in B_{L^i\delta}(y_{-i})=B_{L^i\delta}(f(y_{-i-1}))\subset f(B_{L\cdot L^i\delta}(y_{-i-1}))=f(B_{L^{i+1}\delta}(y_{-i-1}))
\]
and so $x_{-i}=f(x_{-i-1})$ for some $x_{-i-1}\in B_{L^{i+1}\delta}(y_{-i-1})$. By induction on $i$, we obtain a sequence $(x_i)_{i\le0}$ with the desired properties.
\end{proof}

\begin{cor}
Given a continuous map $f\colon X\to X$, let $S=\bigcap_{i\ge0}f^i(X)$ and note that $f(S)=S$. If $f$ is ball expanding, then $S$ is a clopen subset of $X$ and thus $f|_S\colon S\to S$ is ball expanding.
\end{cor}

\begin{proof}
\normalfont
Let $0<L<1$ and $\delta_0>0$ be as in the definition of ball expanding. Given $x\in X$ with $d(x,S)\le\delta_0$, we have $d(x,y_0)\le\delta_0$ for some $y_0\in S$. Since $f(S)=S$, there is a sequence $(y_i)_{i\le0}$ of points in $X$ such that $y_i=f(y_{i-1})$ for all $i\le0$. By Lemma 2.1, we obtain a sequence $(x_i)_{i\le0}$ of points in $X$ such that
\begin{itemize}
\item $x_0=x$,
\item $x_i=f(x_{i-1})$ for all $i\le0$;
\end{itemize}
therefore, $x\in S$. Since $x\in X$ with $d(x,S)\le\delta_0$ is arbitrary, we conclude that $S$ is a clopen subset of $X$, completing the proof.
\end{proof}

\begin{rem}
\normalfont
Let $f\colon X\to X$ be a continuous map and let $S=\bigcap_{i\ge0}f^i(X)$. We can show that if $f$ has the h-shadowing property, then $S$ is a clopen subset of $X$ and thus $f|_S\colon S\to S$ satisfies the h-shadowing property.
\end{rem}

\begin{lem}
Let $f\colon X\to X$ be a continuous map. Let $0<L<1$ and $\delta_0>0$ be as in the definition of ball expanding. For any $x\in X$ and $C\in\mathcal{C}(f)$, if $d(x,C)\le\delta_0$, then $w\rightarrow x$ for all $w\in C$.
\end{lem}

\begin{proof}
Since $d(x,C)\le\delta_0$, we have $d(x,y_0)\le\delta_0$ for some $y_0\in C$. By $f(C)=C$, we obtain a sequence $(y_i)_{i\le0}$ of points in $C$ such that $y_i=f(y_{i-1})$ for all $i\le0$. By Lemma 2.1, there is a sequence $(x_i)_{i\le0}$ of points in $X$ such that
\begin{itemize}
\item $x_0=x$,
\item $x_i=f(x_{i-1})$ for all $i\le0$,
\item $d(x_{-i},y_{-i})\le L^i\delta$ for all $i\ge0$, where $\delta=d(x,y_0)$.
\end{itemize}
These  conditions imply
\[
\lim_{i\to\infty}d(x_{-i},C)=0
\]
and so $v\rightarrow x$ for some $v\in C$. For any $w\in C$, by $v,w\in C$, we obtain $w\rightarrow v$. This with $v\rightarrow x$ implies $w\rightarrow x$, completing the proof.
\end{proof}

\begin{cor}
If a continuous map $f\colon X\to X$ is ball expanding, then every $C\in\mathcal{C}_{\rm ter}(f)$ is a clopen subset of $X$.
\end{cor}

\begin{proof}
Let $0<L<1$ and $\delta_0>0$ be as in the definition of ball expanding. Given $x\in X$ and $C\in\mathcal{C}_{\rm ter}(f)$, if $d(x,C)\le\delta_0$, then by Lemma 2.2, we have $x\in C$. This implies that $C$ is a clopen subset of $X$, proving the corollary.
\end{proof}

\begin{lem}
Let $f\colon X\to X$ be a continuous map. Let $0<L<1$ and $\delta_0>0$ be as in the definition of ball expanding. For any $C\in\mathcal{C}(f)$ and $x\in X$, if $d(x,C)\le\delta_0$ and
\[
\lim_{i\to\infty}d(f^i(x),C)=0,
\]
then $x\in C$.
\end{lem}

\begin{proof}
Let $w\in C$. Since $d(x,C)\le\delta_0$, by Lemma 2.2, $w\rightarrow x$. Since
\[
\lim_{i\to\infty}d(f^i(x),C)=0,
\]
we have $x\rightarrow v$ for some $v\in C$. By $v,w\in C$, we obtain $v\rightarrow w$. This with $x\rightarrow v$ implies $x\rightarrow w$. It follows that $x\rightarrow x$ and $x\leftrightarrow w$; therefore, $x\in C$. This completes the proof of the lemma.
\end{proof}

\begin{cor}
If a continuous map $f\colon X\to X$ is ball expanding, then
\[
X=\bigcup_{i\ge0}f^{-i}(CR(f)).
\]
\end{cor}

\begin{proof}
Let $0<L<1$ and $\delta_0>0$ be as in the definition of ball expanding. For any $x\in X$, we take $C(x,f)\in\mathcal{C}(f)$ such that $\omega(x,f)\subset C(x,f)$ and so
\[
\lim_{i\to\infty}d(f^i(x),C(x,f))=0.
\]
Note that $d(f^i(x),C(x,f))\le\delta_0$ for some $i\ge0$. By Lemma 2.3, we obtain $f^i(x)\in C(x,f)$ and therefore $f^i(x)\in CR(f)$. Since $x\in X$ is arbitrary, the corollary has been proved.
\end{proof}

We say that $X$ is {\em perfect} if $X$ has no isolated points, i.e., $x\in\overline{X\setminus\{x\}}$ for all $x\in X$. By Theorem 2.2 and Corollary 2.2, we obtain the following theorem.

\begin{thm}
If $X$ is perfect and if a continuous map $f\colon X\to X$ is ball expanding, then $h_{\rm top}(f)>0$.
\end{thm}

\begin{proof}
Let $C\in\mathcal{C}_{\rm ter}(f)$. Since $f$ is ball expanding, by Corollary 2.2, $C$ is a clopen subset of $X$. If $h_{\rm top}(f)=0$, then Theorem 2.2 implies that $C$ is a finite set and therefore every $x\in C$ is an isolated point of $X$. Since $X$ is perfect, we conclude that $h_{\rm top}(f)>0$, completing the proof.
\end{proof}

A continuous map $f\colon X\to X$ is said to be {\em chain mixing} if for any $x,y\in X$ and $\delta>0$, there is $j\ge1$ such that for each $k\ge j$, there is a $\delta$-chain $(x_i)_{i=0}^k$ of $f$ with $x_0=x$ and $x_k=y$. If $f$ is mixing, then $f$ is chain mixing, and the converse holds when $f$ has the shadowing property. By Corollary 14 of \cite{RW}, we know that if $X$ is connected and $f$ is chain transitive, i.e., $\mathcal{C}(f)=\{X\}$, then $f$ is chain mixing.

We present the final result of this section (see Appendix B for a generalization). 

\begin{thm}
If $X$ is connected and if a continuous map $f\colon X\to X$ is ball expanding, then $f$ is locally eventually onto.
\end{thm}

\begin{proof}
We take $C\in\mathcal{C}_{\rm ter}(f)$. Since $f$ is ball expanding, by Corollary 2.2, $C$ is a clopen subset of $X$. Since $X$ is connected, it follows that $X=C$. By Corollary 14 of \cite{RW}, $f$ is chain mixing. Since $f$ is ball expanding, $f$ has the h-shadowing property and the shadowing property. By these conditions, we conclude that $f$ is locally eventually onto, completing the proof of the theorem.
\end{proof}

\section{Examples}

In this section, we provide several examples concerning ball expanding maps.

\begin{ex}
\normalfont
Let $X=\{0\}\cup\{2^{-n}\colon n\ge0\}$ and note that $X\subset[0,1]$. We define a surjective continuous map $f\colon X\to X$ by
\begin{itemize}
\item $f(1)=1$,
\item $f(x)=2x$ for all $x\in X\setminus\{1\}$.
\end{itemize}

We shall prove that $f$ is ball expanding.

\begin{proof}
Note that $y=f(2^{-1}y)$ for all $y\in X$. Let $0<\delta<2^{-1}$. It is sufficient to show that for any $x\in X$ and $y\in B_\delta(f(x))\setminus\{f(x)\}$, we have $y=f(z)$ for some $z\in B_{2^{-1}\delta}(x)$. If $x=1$, then $B_\delta(f(x))=B_\delta(1)=\{1\}$ and so $B_\delta(f(x))\setminus\{f(x)\}=\emptyset$. If $x\ne1$, then $f(x)=2x$. Letting $z=2^{-1}y$, we have $y=f(z)$. By
\[
|f(x)-y|=|2x-y|\le\delta,
\]
we obtain
\[
|x-z|=|x-2^{-1}y|\le2^{-1}\delta.
\]
This completes the proof of the claim.
\end{proof}

Note that $CR(f)=\{0,1\}$. Let $X^\mathbb{N}$ be the product space of infinitely many copies of $X$ and let $D$ denote the metric on $X^\mathbb{N}$ defined by
\[
D(x,y)=\sup_{j\ge1}2^{-j}|x_j-y_j|
\]
for all $x=(x_j)_{j\ge1},y=(y_j)_{j\ge1}\in X^\mathbb{N}$. We define $g\colon X^\mathbb{N}\to X^\mathbb{N}$ by: for any $x=(x_j)_{j\ge1},y=(y_j)_{j\ge1}\in X^\mathbb{N}$, $y=g(x)$ if and only if $y_j=f(x_j)$ for all $j\ge1$. Since $f$ is ball expanding, $f$ has the h-shadowing property and so does $g$. Note that $g$ satisfies the following properties
\begin{itemize}
\item $g(x)=x$ for all $x\in\{0,1\}^\mathbb{N}$,
\item $CR(g)=\{0,1\}^\mathbb{N}$, an uncountable set,
\item $\mathcal{C}(g)=\{\{x\}\colon x\in\{0,1\}^\mathbb{N}\}$ and $\mathcal{C}_{\rm ter}(g)=\{\{1^\infty\}\}$.
\end{itemize}
If $g$ is ball expanding, then since $h_{\rm top}(g)=0$, by Theorem 2.2, $CR(g)$ must be a finite set. It follows that $g$ is not ball expanding. Defining $x=(x_j)_{j\ge1}\in X^\mathbb{N}$ by $x_j=2^{-j+1}$ for all $j\ge1$, we easily see that
\[
\lim_{i\to\infty}D(g^i(x),1^\infty)=0;
\]
however,
\[
x\not\in\bigcup_{i\ge0}g^{-i}(CR(g))
\] 
(see Corollary 2.3 in Section 2).
\end{ex}

\begin{ex}
\normalfont
Let $I_n=[4^{-n},2\cdot 4^{-n}]$ for all $n\ge1$. Let
\[
X=\{0,2\}\cup\bigcup_{n\ge1}I_n
\]
and note that $X\subset[0,2]$. We define a surjective continuous map $f\colon X\to X$ by
\begin{itemize}
\item $f(x)=2$ for all $x\in\{2\}\cup I_1$,
\item $f(x)=4x$ for all $x\in X\setminus(\{2\}\cup I_1)$.
\end{itemize}

We shall prove that $f$ is ball expanding.

\begin{proof}
Note that $y=f(4^{-1}y)$ for all $y\in X$. Let $0<\delta<3\cdot2^{-1}$. It is sufficient to show that for any $x\in X$ and $y\in B_\delta(f(x))\setminus\{f(x)\}$, we have $y=f(z)$ for some $z\in B_{4^{-1}\delta}(x)$. If $x\in\{2\}\cup I_1$, then $B_\delta(f(x))=B_\delta(2)=\{2\}$ and so $B_\delta(f(x))\setminus\{f(x)\}=\emptyset$. If $x\not\in\{2\}\cup I_1$, then $f(x)=4x$. Letting $z=4^{-1}y$, we have $y=f(z)$. By
\[
|f(x)-y|=|4x-y|\le\delta,
\]
we obtain
\[
|x-z|=|x-4^{-1}y|\le4^{-1}\delta,
\]
completing the proof.
\end{proof}

Note that $X$ is uncountable and not totally disconnected, but $h_{\rm top}(f)=0$.
\end{ex}

\begin{ex}
\normalfont
In \cite{BGO}, the tent map $T\colon[0,1]\to [0,1]$, defined as $T(x)=1-|1-2x|$ for all $x\in[0,1]$, is given as an example of a ball expanding map (see Example 2.9 of \cite{BGO}). For the sake of completeness, we give a proof.

\begin{proof}
It is sufficient to show that 
\[
B_{2\delta}(T(x))\subset T(B_\delta(x))
\]
for all $x\in[0,1]$ and $0<\delta\le4^{-1}$. We may assume $x\in[0,2^{-1}]$. Note that $T(x)=2x$. If $x\in[0,\delta]$, then $T(B_\delta(x))=[0,2x+2\delta]$. Since
\[
B_{2\delta}(T(x))\subset[2x-2\delta,2x+2\delta]\cap[0,1]=[0,2x+2\delta],
\]
we obtain $B_{2\delta}(T(x))\subset T(B_\delta(x))$. If $x\in[\delta,2^{-1}-\delta]$, then $T(B_\delta(x))=[2x-2\delta,2x+2\delta]$. Since
\[
B_{2\delta}(T(x))\subset[2x-2\delta,2x+2\delta],
\]
we obtain $B_{2\delta}(T(x))\subset T(B_\delta(x))$. If $x\in[2^{-1}-\delta,2^{-1}]$, then $T(B_\delta(x))=[2x-2\delta,1]$. Since
\[
B_{2\delta}(T(x))\subset[2x-2\delta,2x+2\delta]\cap[0,1]=[2x-2\delta,1],
\]
we obtain $B_{2\delta}(T(x))\subset T(B_\delta(x))$. This completes the proof of the claim.
\end{proof}
Note that $[0,1]$ is connected and $T$ is locally eventually onto.

Let $S^1=\{z\in\mathbb{C}\colon|z|=1\}$ with the arc-length metric. It is easy to see that the doubling map $f\colon S^1\to S^1$, defined as $f(z)=z^2$ for all $z\in S^1$, is ball expanding, where $S^1$ is connected and $f$ is locally eventually onto.

Let $g\colon[0,1]\to[0,1]$ be the logistic map defined as $g(x)=4x(1-x)$ for all $x\in[0,1]$. It is well-known that $g\colon[0,1]\to[0,1]$ is topologically conjugate to $T\colon[0,1]\to [0,1]$. As claimed in \cite{BGO}, since $T$ is ball expanding and so has the h-shadowing property, $g$ also has the h-shadowing property. Let $0<\delta\le2^{-1}$ and note that $g(2^{-1}-\delta)=g(2^{-1}+\delta)=1-4\delta^2$. We have
\[
B_\delta(g(2^{-1}))=B_\delta(1)=[1-\delta,1]
\]
and
\[
g(B_\delta(2^{-1}))=g([2^{-1}-\delta,2^{-1}+\delta])=[1-4\delta^2,1].
\]
If $\delta<4^{-1}$, then $1-\delta<1-4\delta^2$ and so $B_\delta(g(2^{-1}))\not\subset g(B_\delta(2^{-1}))$. This implies that $g$ is not ball expanding (see Example 5.5 of \cite{BGO}).
\end{ex}

\begin{ex}
\normalfont
Let $X^\mathbb{N}$ be the product space of infinitely many copies of $X$ and let $D$ denote the metric on $X^\mathbb{N}$ defined by
\[
D(x,y)=\sup_{j\ge1}2^{-j}d(x_j,y_j)
\]
for all $x=(x_j)_{j\ge1},y=(y_j)_{j\ge1}\in X^\mathbb{N}$. We consider the shift map $f\colon X^\mathbb{N}\to X^\mathbb{N}$ defined as: for any $x=(x_j)_{j\ge1},y=(y_j)_{j\ge1}\in X^\mathbb{N}$, $y=f(x)$ if and only if $y_j=x_{j+1}$ for all $j\ge1$.

We shall prove that $f$ is ball expanding.

\begin{proof}
It is sufficient to show that 
\[
B_\delta(f(x))\subset f(B_{2^{-1}\delta}(x))
\]
for all $x\in X^\mathbb{N}$ and $\delta>0$. Given any $x\in X^\mathbb{N}$ and $y\in B_\delta(f(x))$, we have $d(x_{j+1},y_j)\le2^j\delta$ for all $j\ge1$. We define $z=(z_j)_{j\ge1}\in X^\mathbb{N}$ by
\begin{itemize}
\item $z_1=x_1$,
\item $z_{j+1}=y_j$ for all $j\ge1$.
\end{itemize}
Note that $y=f(z)$. Letting $m_j=2^{-j-1}d(x_{j+1},z_{j+1})$, $j\ge0$, we see that
 \begin{itemize}
\item $m_0=2^{-1}d(x_1,x_1)=0$,
\item $m_j=2^{-j-1}d(x_{j+1},y_j)\le2^{-j-1}\cdot2^j\delta\le2^{-1}\delta$ for all $j\ge1$.
\end{itemize}
It follows that $D(x,z)=\sup_{j\ge0}m_j\le2^{-1}\delta$, i.e., $z\in B_{2^{-1}\delta}(x)$. Since $y\in B_\delta(f(x))$ is arbitrary, we conclude that
\[
B_\delta(f(x))\subset f(B_{2^{-1}\delta}(x)).
\]
Since $x\in X$ is arbitrary, this completes the proof of the claim.
\end{proof}

Note that $f\colon X^\mathbb{N}\to X^\mathbb{N}$ is locally eventually onto. If $X=\{0,1\}$, then $X^\mathbb{N}$ is perfect and totally disconnected. If $X=[0,1]$, then $X^\mathbb{N}$ is connected and infinite-dimensional.
\end{ex}

\begin{ex}
\normalfont
We shall give an example of a continuous map $f\colon X\to X$ such that
\begin{itemize}
\item $f$ is expanding,
\item $f$ is not open,
\item $f$ is locally eventually onto.
\end{itemize}
Let $D$ be the metric on $\{-1,0,1\}^\mathbb{N}$ defined by
\[
D(x,y)=\sup_{j\ge1}2^{-j}|x_j-y_j|
\]
for all $x=(x_j)_{j\ge1},y=(y_j)_{j\ge1}\in\{-1,0,1\}^\mathbb{N}$. We define a map $g\colon\{-1,0,1\}^\mathbb{N}\to\{-1,0,1\}^\mathbb{N}$ by: for any $x=(x_j)_{j\ge1},y=(y_j)_{j\ge1}\in\{-1,0,1\}^\mathbb{N}$, $y=g(x)$ if and only if
\[
y_j=(-1)^{x_1}x_{j+1}
\]
for all $j\ge1$. Given  $x=(x_j)_{j\ge1},y=(y_j)_{j\ge1}\in\{-1,0,1\}^\mathbb{N}$,  if $0<D(x,y)<2^{-1}$, then there is $n\ge1$ such that
\begin{itemize}
\item $x_j=y_j$ for all $1\le j\le n$,
\item $x_{n+1}\ne y_{n+1}$.
\end{itemize}
We obtain
\[
D(x,y)\le 2^{-n-1}|x_{n+1}-y_{n+1}|
\]
and
\[
D(g(x),g(y))\ge 2^{-n}|x_{n+1}-y_{n+1}|;
\]
therefore, $D(g(x),g(y))\ge2D(x,y)$. It follows that that $g$ is expanding. Let
\[
X=\{-1,0\}^\mathbb{N}\cup\{0,1\}^\mathbb{N},
\]
a closed subset of $\{-1,0,1\}^\mathbb{N}$ with $g(X)=X$, and $f=g|_X\colon X\to X$. Since $g$ is expanding, so is $f$. Let
\[
U=\{x=(x_j)_{j\ge1}\in\{0,1\}^\mathbb{N}\colon x_1=1\},
\]
an open subset of $X$ with $f(U)=\{-1,0\}^\mathbb{N}$. Letting $0^\infty=(0,0,0,\dots)$, by
\[
\{0^\infty\}=\{-1,0\}^\mathbb{N}\cap\{0,1\}^\mathbb{N},
\]
we see that $\{-1,0\}^\mathbb{N}$ is not open in $X$; therefore, $f$ is not an open map. It remains to show that $f$ is locally eventually onto.

Given  $x=(x_j)_{j\ge1}\in\{0,1\}^\mathbb{N}$, $n\ge1$, $z\in\{0,1\}$, $y=(y_j)_{j\ge1}\in\{0,1\}^\mathbb{N}$, let
\[
x'=(x_1,x_2,\dots,x_n,z,y_1,y_2,y_3,\dots)
\]
and note that
\[
f^n(x')=(z,y_1,y_2,y_3,\dots)
\]
or
\[
f^n(x')=(-z,-y_1,-y_2,-y_3,\dots).
\]
If $f^n(x')=(z,y_1,y_2,y_3,\dots)$, then
\begin{equation*}
f^{n+1}(x')=
\begin{cases}
(y_1,y_2,y_3,\dots)&\text{if $z=0$}\\
(-y_1,-y_2,-y_3,\dots)&\text{if $z=1$}
\end{cases}
.
\end{equation*}
If $f^n(x')=(-z,-y_1,-y_2,-y_3,\dots)$, then
\begin{equation*}
f^{n+1}(x')=
\begin{cases}
(-y_1,-y_2,-y_3,\dots)&\text{if $z=0$}\\
(y_1,y_2,y_3,\dots)&\text{if $z=1$}
\end{cases}
.
\end{equation*}
It follows that $f^{n+1}(V)=X$, where
\[
V=\{x'=(x'_j)_{j\ge1}\in\{0,1\}^\mathbb{N}\colon x_j=x'_j\:\text{ for all $1\le j\le n$}\}.
\]

Similarly, given  $x=(x_j)_{j\ge1}\in\{-1,0\}^\mathbb{N}$, $n\ge1$, $z\in\{-1,0\}$, $y=(y_j)_{j\ge1}\in\{-1,0\}^\mathbb{N}$, let
\[
x'=(x_1,x_2,\dots,x_n,z,y_1,y_2,y_3,\dots)
\]
and note that
\[
f^n(x')=(z,y_1,y_2,y_3,\dots)
\]
or
\[
f^n(x')=(-z,-y_1,-y_2,-y_3,\dots).
\]
If $f^n(x')=(z,y_1,y_2,y_3,\dots)$, then
\begin{equation*}
f^{n+1}(x')=
\begin{cases}
(-y_1,-y_2,-y_3,\dots)&\text{if $z=-1$}\\
(y_1,y_2,y_3,\dots)&\text{if $z=0$}
\end{cases}
.
\end{equation*}
If $f^n(x')=(-z,-y_1,-y_2,-y_3,\dots)$, then
\begin{equation*}
f^{n+1}(x')=
\begin{cases}
(y_1,y_2,y_3,\dots)&\text{if $z=-1$}\\
(-y_1,-y_2,-y_3,\dots)&\text{if $z=0$}
\end{cases}
.
\end{equation*}
It follows that $f^{n+1}(W)=X$, where
\[
W=\{x'=(x'_j)_{j\ge1}\in\{-1,0\}^\mathbb{N}\colon x_j=x'_j\:\text{ for all $1\le j\le n$}\}.
\]

From the above, we conclude that $f$ is locally eventually onto.
\end{ex}

\section{A related result on the $h$-expansiveness}

A continuous map $f\colon X\to X$ is said to be {\em transitive} if for any non-empty open subsets $U,V$ of $X$, there is $i\ge0$ such that $f^i(U)\cap V\ne\emptyset$. The notion of {\em $h$-expansiveness} was introduced in \cite{B}. By Theorem 1.4 of \cite{K2}, we know that if a continuous map $f\colon X\to X$ is transitive and has the shadowing property, then the following conditions are equivalent
\begin{itemize}
\item $f$ is not $h$-expansive,
\item for any $\epsilon>0$, there is $r>0$ such that for every $\delta>0$, there is a pair
\[
((x_i)_{i=0}^k,(y_i)_{i=0}^k)
\]
of $\delta$-chains of $f$ with $(x_0,x_k)=(y_0,y_k)$ and
\[
r\le\max_{0\le i\le k}d(x_i,y_i)\le\epsilon.
\]
\end{itemize}
We say that a map $f\colon X\to X$ is {\em locally one-to-one} if there is $\epsilon>0$ such that $x\ne y$ implies $f(x)\ne f(y)$ for all $x,y\in X$ with $d(x,y)\le\epsilon$.  As mentioned in Remark 1.1, by Theorem 2.5 of \cite{BGO}, we know that for a continuous map $f\colon X\to X$, the following conditions are equivalent
\begin{itemize}
\item $f$ is expanding and open,
\item $f$ is ball expanding and locally one-to-one.
\end{itemize}

The aim of this section is to prove the following theorem.

\begin{thm}
Given a continuous map $f\colon X\to X$, if
\begin{itemize}
\item $f$ is transitive,
\item $f$ is ball expanding,
\item $f$ is $h$-expansive,
\end{itemize}
then $f$ is locally one-to-one and so expanding.
\end{thm}

\begin{proof}
Let $0<L<1$ and $\delta_0>0$ be as in the definition of ball expanding.  Assume for contradiction that $f$ is not locally one-to-one. Then, given any $\epsilon>0$, there are $x,y\in X$ such that
\[
0<d(x,y)\le\min\{\delta_0,\epsilon\}
\]
and $f(x)=f(y)$. Let $r=d(x,y)$. Since $f$ is transitive and so surjective, there is a  sequence $(y_i)_{i\le0}$ of points in $X$ such that
\begin{itemize}
\item $y_0=y$,
\item $y_i=f(y_{i-1})$ for all $i\le0$.
\end{itemize}
Since $r\le\delta_0$, by Lemma 2.1, we obtain a sequence $(x_i)_{i\le0}$ of points in $X$ such that
\begin{itemize}
\item $x_0=x$,
\item $x_i=f(x_{i-1})$ for all $i\le0$,
\item $d(y_{-i},x_{-i})\le L^i r$ for all $i\ge0$.
\end{itemize}
Note that
\[
\sup_{i\ge0}d(y_{-i},x_{-i})=r\le\epsilon
\]
and
\[
\lim_{i\to\infty}d(y_{-i},x_{-i})=0.
\]
Given $\delta>0$, if $k\ge1$ is sufficiently large, then letting
\[
(z_i)_{i=0}^{k+2}=(y_{-k-1},x_{-k},\dots,x_{-1},x_0,f(x))
\]
and
\[
(w_i)_{i=0}^{k+2}=(y_{-k-1},y_{-k},\dots,y_{-1},y_0,f(y)),
\]
we see that
\[
((z_i)_{i=0}^{k+2},(w_i)_{i=0}^{k+2})
\]
is a pair of $\delta$-chains of $f$ with $(z_0,z_{k+2})=(w_0,w_{k+2})$ and
\[
r=\max_{0\le i\le k+2}d(z_i,w_i)\le\epsilon.
\]
By Theorem 1.4 of \cite{K2}, we conclude that $f$ is not $h$-expansive, a contradiction; therefore, the theorem has been proved.
\end{proof}

\begin{rem}
\normalfont
\begin{itemize}
\item  The map  $f\colon X\to X$ given in Example 3.2 has the following properties
\begin{itemize}
\item $f$ is not transitive,
\item $f$ is ball expanding,
\item $h_{\rm top}(f)=0$ and therefore $f$ is $h$-expansive,
\item $f$ is not locally one-to-one (expanding).
\end{itemize}
This shows that the transitivity assumption in Theorem 4.1 is necessary.
\item The tent map  $T\colon[0,1]\to[0,1]$ defined in Example 3.3 has the following properties
\begin{itemize}
\item $T$ is locally eventually onto and thus transitive,
\item $T$ is ball expanding,
\item $T$ is not locally one-to-one (expanding).
\end{itemize}
As a consequence of Theorem 4.1, $T$ is not $h$-expansive.
\item Let $X=[0,1]$. The map  $f\colon X^\mathbb{N}\to X^\mathbb{N}$ given in Example 3.4  has the following properties
\begin{itemize}
\item $f$ is locally eventually onto and thus transitive,
\item $f$ is ball expanding,
\item $h_{\rm top}(f)=\infty$ and therefore $f$ is not $h$-expansive,
\item $f$ is not locally one-to-one (expanding).
\end{itemize}
This shows that the assumption of $h$-expansiveness in Theorem 4.1 is necessary.
\end{itemize}
\end{rem}

\section{A related result on the L-shadowing property}

We shall recall the definition of {\em L-shadowing} from \cite{ACCV}. Let $Y$ be a compact metric space endowed with a metric $D$ and let $g\colon Y\to Y$ be a homeomorphism. Let $\xi=(x_i)_{i\in\mathbb{Z}}$ be a sequence of points in $Y$. For $\delta>0$, $\xi$ is called a {\em $\delta$-limit-pseudo orbit} of $g$ if
\[
\sup_{i\in\mathbb{Z}}D(g(x_i),x_{i+1})\le\delta
\]
and
\[
\lim_{i\to-\infty}D(g(x_i),x_{i+1})=\lim_{i\to+\infty}D(g(x_i),x_{i+1})=0.
\]
For $\epsilon>0$, $\xi$ is said to be {\em $\epsilon$-limit shadowed} by $x\in Y$ if
\[
\sup_{i\in\mathbb{Z}}D(x_i,g^i(x))\le\epsilon
\]
and
\[
\lim_{i\to-\infty}D(x_i,g^i(x))=\lim_{i\to+\infty}D(x_i,g^i(x))=0.
\]
We say that $g$ has the {\em L-shadowing property} if for any $\epsilon>0$, there is $\delta>0$ such that every $\delta$-limit-pseudo orbit of $g$ is $\epsilon$-limit shadowed by some point of $Y$. For $x\in Y$ and $r>0$, let
\[
V_r^s(x)=\{z\in Y\colon\text{$\sup_{i\ge0}D(g^i(x),g^i(z))\le r$ and $\lim_{i\to+\infty}D(g^i(x),g^i(z))=0$}\},
\]
\[
V_r^u(x)=\{z\in Y\colon\text{$\sup_{i\le0}D(g^i(x),g^i(z))\le r$ and $\lim_{i\to-\infty}D(g^i(x),g^i(z))=0$}\}.
\]
By Theorem E of \cite{ACCV}, we know that the following conditions are equivalent
\begin{itemize}
\item $g$ has the L-shadowing property,
\item $g$ has the shadowing property; and for any $\epsilon>0$, there is $\delta>0$ such that $D(x,y)\le\delta$ implies $V_\epsilon^s(x)\cap V_\epsilon^u(y)\ne\emptyset$ for all $x,y\in Y$.
\end{itemize}

Let $X^\mathbb{Z}$ be the product space of infinitely many copies of $X$ and let $D$ denote the metric on $X^\mathbb{Z}$ defined by
\[
D(x,y)=\sup_{i\in\mathbb{Z}}2^{-|i|}d(x_i,y_i)
\]
for all $x=(x_i)_{i\in\mathbb{Z}},y=(y_i)_{i\in\mathbb{Z}}\in X^\mathbb{Z}$.  Given a surjective continuous map $f\colon X\to X$, let
\[
Y=\{(x_i)_{i\in\mathbb{Z}}\in X^\mathbb{Z}\colon\text{$x_{i+1}=f(x_i)$ for all $i\in\mathbb{Z}$}\}
\]
and let $g\colon Y\to Y$ be the homeomorphism defined as: for any $x=(x_i)_{i\in\mathbb{Z}},y=(y_i)_{i\in\mathbb{Z}}\in Y$, $y=g(x)$ if and only if $y_i=f(x_i)(=x_{i+1})$ for all $i\in\mathbb{Z}$. 

The aim of this section is to prove the following theorem.

\begin{thm}
Let $f\colon X\to X$ be a surjective continuous map and let $g\colon Y\to Y$ be the homeomorphism defined as above. If $f$ is ball expanding, then $g$ has the L-shadowing property.
\end{thm}

\begin{proof}
Since $f$ is ball expanding, $f$ has the h-shadowing property and so the shadowing property. Since $f$ has the shadowing property, so does $g$ (see, e.g., Theorem 2.3.8 of \cite{AH}). By Theorem E of \cite{ACCV}, it suffices to show that  for any $\epsilon>0$, there is $\delta>0$ such that $D(x,y)\le\delta$ implies $V_\epsilon^s(x)\cap V_\epsilon^u(y)\ne\emptyset$ for all $x,y\in Y$. Let $0<L<1$ and $\delta_0>0$ be as in the definition of ball expanding. Let $0<\epsilon\le2\delta_0$. Since $f$ has the h-shadowing property, there is $0<\delta\le\epsilon/2$ such that for any $\delta$-chain $(z_i)_{i=0}^k$ of $f$, we have
\[
\sup_{0\le i\le k-1}d(z_i,f^i(z))\le\epsilon/2
\]
and $z_k=f^k(z)$ for some $z\in X$. We fix $n\ge1$ such that $|i|\ge n$ implies $2^{-|i|}{\rm diam}X\le\epsilon$. Assume that $D(x,y)\le2^{-n}\delta$ for some $x=(x_i)_{i\in\mathbb{Z}},y=(y_i)_{i\in\mathbb{Z}}\in Y$. Then, for each $-n\le i\le n$, we have
\[
2^{-|i|}d(x_i,y_i)\le D(x,y)\le2^{-n}\delta
\]
and so
\[
d(x_i,y_i)\le2^{-n+|i|}\delta\le\delta.
\]
Since
\[
\gamma=(y_{-n},\dots,y_{-1},y_0,\dots,y_{n-1},x_n)
\]
is a $\delta$-chain of $f$, we have
\[
\sup_{-n\le i\le n-1}d(y_i,f^{i+n}(z))\le\epsilon/2
\]
and $x_n=f^{2n}(z)$ for some $z\in X$. As $d(y_{-n},z)\le\epsilon/2\le\delta_0$, by Lemma 2.1, we obtain a sequence $(z_{-n+i})_{i\le0}$ of points in $X$ such that
\begin{itemize}
\item $z_{-n}=z$,
\item $z_{-n+i}=f(z_{-n+i-1})$ for all $i\le0$,
\item $d(y_{-n-i},z_{-n-i})\le L^i\epsilon/2$ for all $i\ge0$.
\end{itemize}
Let $z_i=f^{i+n}(z)$ for every $i\ge-n$ and note that $\tilde{z}=(z_i)_{i\in\mathbb{Z}}\in Y$. If $-n\le i\le n-1$, then
\[
d(x_i,z_i)=d(x_i,f^{i+n}(z))\le d(x_i,y_i)+d(y_i,f^{i+n}(z))\le\delta+\epsilon/2\le\epsilon.
\]
If $i\ge n$, then
\[
d(x_i,z_i)=d(f^{i-n}(x_n),f^{i+n}(z))=d(f^{i-n}(x_n),f^{i-n}(f^{2n}(z)))=d(f^{i-n}(x_n),f^{i-n}(x_n))=0.
\]
It follows that
\[
\sup_{i\ge0}D(g^i(x),g^i(\tilde{z}))\le\epsilon
\]
and
\[
\lim_{i\to+\infty}D(g^i(x),g^i(\tilde{z}))=0;
\]
therefore, $\tilde{z}\in V_\epsilon^s(x)$.  If $-n\le i\le n-1$, then
\[
d(y_i,z_i)=d(y_i,f^{i+n}(z))\le\epsilon/2.
\]
Since $d(y_{-n-i},z_{-n-i})\le\epsilon/2$ for all $i\ge0$ and
\[
\lim_{i\to+\infty}d(y_{-n-i},z_{-n-i})=0,
\]
we obtain
\[
\sup_{i\le0}D(g^i(y),g^i(\tilde{z}))\le\epsilon
\]
and
\[
\lim_{i\to-\infty}D(g^i(y),g^i(\tilde{z}))=0;
\]
therefore, $\tilde{z}\in V_\epsilon^u(y)$. By $\tilde{z}\in V_\epsilon^s(x)\cap V_\epsilon^u(y)$, we conclude that $V_\epsilon^s(x)\cap V_\epsilon^u(y)\ne\emptyset$. Since $0<\epsilon\le2\delta_0$ is arbitrary, this completes the proof of the theorem.
\end{proof}

\begin{rem}
\normalfont
Let $X=[0,1]$ and let $f\colon X\to X$ be the tent map defined as in Example 3.3. Since $f$ is ball expanding, by Theorem 5.1, $g\colon Y\to Y$ has the L-shadowing property.
\end{rem}

\appendix

\section{}

The aim of this Appendix A is to show that a similar conclusion as in Lemma 2.1 holds under the assumption of $L$-Lipschitz shadowing, where $0<L<1$. Precisely, we prove the following. 

\begin{lem}
Let $f\colon X\to X$ be a continuous map. Let $0<L<1$ and $\delta_0>0$ be as in the definition of $L$-Lipschitz shadowing property. For any $x\in X$ and any sequence $(y_i)_{i\le0}$ of points in $X$, if
\begin{itemize}
\item $d(y_0,x)\le\delta_0$,
\item $y_i=f(y_{i-1})$ for all $i\le0$,
\end{itemize}
then there is a sequence $(x_i)_{i\le0}$ of points in $X$ such that
\begin{itemize}
\item $x_0=x$,
\item $x_i=f(x_{i-1})$ for all $i\le0$,
\item
\[
d(y_{-i},x_{-i)}\le\frac{L^i}{1-L}\delta
\]
for all $i\ge1$, where $\delta=d(y_0,x)$.
\end{itemize}
\end{lem}

\begin{proof}
Let $x_0=x$. Let us show that there is a sequence $(z_j)_{j\le0}$ of points in $X$ with the following property: for any $j\ge1$ and $1\le i\le j-1$,
\begin{itemize}
\item $\max\{d(y_{-j},z_{-j}),d(x_0,f^j(z_{-j}))\}\le L^j\delta$,
\item $d(f^{j-1-i}(z_{-j+1}),f^{j-i}(z_{-j}))\le L^j\delta$.
\end{itemize}
We argue by induction on $j$. When $j=1$, since $\gamma_1=(y_{-1},x_0)$ is a $\delta$-chain of $f$, we have
\[
\max\{d(y_{-1},z_{-1}),d(x_0,f(z_{-1}))\}\le L\delta
\]
for some $z_{-1}\in X$. Assume that
\[
\max\{d(y_{-j},z_{-j}),d(x_0,f^j(z_{-j}))\}\le L^j\delta
\]
for some $j\ge1$ and $z_{-j}\in X$. Then, since
\[
\gamma_j=(y_{-j-1},z_{-j},f(z_{-j}),\dots,f^{j-1}(z_{-j}),x_0)
\]
is an $L^j\delta$-chain of $f$, there is $z_{-j-1}\in X$ such that
\begin{itemize}
\item $\max\{d(y_{-j-1},z_{-j-1}),d(x_0,f^{j+1}(z_{-j-1})) \}\le L^{j+1}\delta$,
\item $d(f^{j-i}(z_{-j}),f^{j+1-i}(z_{-j-1}))\le L^{j+1}\delta$ for all $1\le i\le j$;
\end{itemize}
therefore, the induction is complete.
 
Given $i\ge1$, since
\[
\sum_{j=i}^\infty d(f^{j-i}(z_{-j}),f^{j+1-i}(z_{-j-1}))\le\sum_{j=i}^\infty L^{j+1}\delta=\frac{L^{i+1}}{1-L}\delta<\infty,
\]
we have
\[
\lim_{j\to\infty}f^{j-i}(z_{-j})=x_{-i}
\]
for some $x_{-i}\in X$. We have also
\[
\lim_{j\to\infty}f^j(z_{-j})=x_0.
\]
For any $i\ge0$,
\[
f(x_{-i-1})=f(\lim_{j\to\infty}f^{j-i-1}(z_{-j}))=\lim_{j\to\infty}f^{j-i}(z_{-j})=x_{-i}.
\]
For any $i\ge1$ and $j\ge i+1$,
\begin{align*}
d(y_{-i},f^{j-i}(z_{-j}))&\le d(y_{-i},z_{-i})+\sum_{k=i+1}^j d(f^{k-1-i}(z_{-k+1}),f^{k-i}(z_{-k}))\\
&\le L^i\delta+\sum_{k=i+1}^j L^k\delta\le\sum_{k=i}^\infty L^k\delta=\frac{L^i}{1-L}\delta;
\end{align*}
therefore, letting $j\to\infty$, we obtain
\[
d(y_{-i},x_{-i)}\le\frac{L^i}{1-L}\delta.
\]
This completes the proof of the lemma.
\end{proof}

\section{}

By Theorem 4.3 of \cite{BGO}, we know that if a continuous map $f\colon X\to X$ is ball expanding, then $f$ has the h-shadowing property. The aim of this Appendix B is to prove the following generalization of Theorem 2.5.

\begin{thm}
If $X$ is connected and if a continuous map $f\colon X\to X$ has the h-shadowing property, then $f$ is locally eventually onto. 
\end{thm}

Let us first prove the following lemma.

\begin{lem}
Let $f\colon X\to X$ be a continuous map. For a closed subset $S$ of $X$, if
\begin{itemize}
\item $f(S)=S$,
\item $S$ is chain stable,
\end{itemize}
then for any $r>0$, there is a subset $A$ of $X$ such that $S\subset A\subset B_r(S)$ and $f(\overline{A})\subset{\rm int}[A]$, where ${\rm int}[A]$ denotes the interior of $A$.
\end{lem}

\begin{proof}
For $\delta>0$, let $A$ be the set of $y\in X$ for which there are $x\in S$ and a $\delta$-chain $(x_i)_{i=0}^k$ of $f$ such that $x_0=x$ and $x_k=y$. For any $y\in S$, since $f(S)=S$, by taking $x\in S$ with $y=f(x)$, we have that $(x,y)$ is a $\delta$-chain of $f$, thus $y\in A$. It follows that $S\subset A$. Since $S$ is chain stable, for any $r>0$, we have $A\subset B_r(S)$ for some $\delta>0$.

For any $z\in\overline{A}$, we take $b>0$ such that $f(B_b(z))\subset B_{\delta/2}(f(z))$. Since $z\in\overline{A}$, we have $y\in B_b(z)$ for some $y\in A$. Since $y\in A$, there are $x\in S$ and a $\delta$-chain $(x_i)_{i=0}^k$ of $f$ such that $x_0=x$ and $x_k=y$. For any $w\in B_{\delta/2}(f(z))$, by
\[
d(f(y),w)\le d(f(y),f(z))+d(f(z),w)\le\delta/2+\delta/2=\delta,
\]
we see that
\[
\gamma=(x_0,x_1,\dots,x_k,w)
\]
is a $\delta$-chain of $f$; therefore, $w\in A$. It follows that $B_{\delta/2}(f(z))\subset A$, thus $f(z)\in{\rm int}[A]$. Since $z\in\overline{A}$ is arbitrary, we obtain $f(\overline{A})\subset{\rm int}[A]$, completing the proof of the lemma.
\end{proof}

With the aid of Lemma B.1, we prove the following lemma.

\begin{lem}
Let $f\colon X\to X$ be a continuous map with the h-shadowing property. For a closed subset $S$ of $X$, if
\begin{itemize}
\item $f(S)=S$,
\item $S$ is chain stable,
\end{itemize}
then for any $r>0$, there is a clopen subset $B$ of $X$ such that $S\subset B\subset B_r(S)$ and $f(B)=B$.
\end{lem}

\begin{proof}
For $r>0$, we take a subset $A$ of $X$ as given in Lemma B.1. Note that $f(\overline{A})\subset\overline{A}$ and let $B=\bigcap_{i\ge0}f^i(\overline{A})$. We have the following properties:
\begin{itemize}
\item $S\subset B\subset{\rm int}[A]\subset  B_r(S)$,
\item $B$ is closed in $X$ and satisfies $f(B)=B$.
\end{itemize}
Let us prove that $B$ is open (and so clopen) in $X$. We take $\epsilon>0$ such that $B_\epsilon(B)\subset A$. For any $a>0$, by compactness, there is $i\ge0$ such that $f^j(\overline{A})\subset B_a(B)$ for all $j\ge i$. It follows that $f^j(B_\epsilon (B))\subset B_a(B)$ for all $j\ge i$. For this $\epsilon$, we fix $\delta>0$ as in the definition of the h-shadowing property. If $B$ is not open in $X$, then we have
\[
0<d(x,B)\le\delta
\]
for some $x\in X$. We take $a>0$ and $j\ge1$ with $a<d(x,B)$ and  $f^j(B_\epsilon(B))\subset B_a(B)$. There are $y,z\in B$ such that 
\[
d(x,y)=d(x,B)\le\delta
\]
and $y=f^j(z)$. Since
\[
\gamma=(z,f(z),\dots,f^{j-1}(z),x)
\]
is a $\delta$-chain of $f$, we obtain $x=f^j(w)$ for some $w\in B_\epsilon(z)$. It follows that
\[
x\in f^j(B_\epsilon(B))\subset B_a(B),
\]
i.e., $d(x,B)\le a$, a contradiction. Thus, we conclude that $B$ is open in $X$, completing the proof of the lemma.
\end{proof}

Given a continuous map $f\colon X\to X$ and $x\in X$, we define the {\em chain $\omega$-limit set} $\omega^\ast(x,f)$ of $x$ for $f$ as the set of $y\in X$ such that for any $\delta>0$ and $l>0$, there is a $\delta$-chain $(x_i)_{i=0}^k$ of $f$ with $x_0=x$, $x_k=y$, and $l\le k$. Note that for every $x\in X$,
\begin{itemize}
\item $\omega^\ast(x,f)$ is closed in $X$,
\item $f(\omega^\ast(x,f))=\omega^\ast(x,f)$,
\item $\omega^\ast(x,f)$ is chain stable, because $y\rightarrow z$ implies $z\in\omega^\ast(x,f)$ for all $y\in\omega^\ast(x,f)$ and $z\in X$.
\end{itemize}

Finally, we prove Theorem B.1.

\begin{proof}[Proof of Theorem B.1]
Let $x\in X$. Since $f$ has the h-shadowing property, Lemma B.2 implies that for any $r>0$, there is a clopen subset $B$ of $X$ such that
\[
\omega^\ast(x,f)\subset B\subset B_r(\omega^\ast(x,f)).
\]
As $X$ is connected, we obtain $B=X$ and so $X=B_r(\omega^\ast(x,f))$. Since $r>0$ is arbitrary, we obtain $X=\omega^\ast(x,f)$. Since $x\in X$ is arbitrary, by the definition of the chain $\omega$-limit set, we obtain $x\rightarrow y$ for all $x,y\in X$; therefore, $\mathcal{C}(f)=\{X\}$. By Corollary 14 of \cite{RW}, $f$ is chain mixing. Since $f$ has the h-shadowing property, we conclude that $f$ is locally eventually onto, completing the proof of the theorem.
\end{proof}

\end{document}